\documentclass[DIV=11,a4paper]{artclcomp}

\usepackage{times}
\usepackage{calc}
\usepackage{enumitem}
\usepackage{eurosym}
\usepackage{textcomp}
\usepackage{mathrsfs}
\usepackage{graphicx}
\usepackage{enumitem}
\usepackage{stmaryrd}
\usepackage{epsfig}
\usepackage[usenames,dvipsnames]{xcolor}

\newcommand{\R}{\mathbb{R}}

\newcommand*{\dom}{{\mathrm{dom}}}

\usepackage[hyperref,amsmath,thmmarks]{ntheorem}

\theoremseparator{.}

\theoremsymbol{\ensuremath{\blacklozenge}}
\theoremheaderfont{\itshape}

\numberwithin{equation}{section}

\usepackage{lipsum}
\usepackage{todonotes}
\usepackage{graphicx}
\usepackage{amsfonts}
\usepackage{accents}
\usepackage{xspace}
\usepackage{dsfont}
\usepackage{mathtools}
\usepackage[capitalize]{cleveref}

\usepackage{
	comment,
	thmtools, 
	enumitem, 
	pgfplots, 
	float, 
	wrapfig
}

\usetikzlibrary{angles,quotes,patterns,arrows,decorations.markings}

\allowdisplaybreaks[1]

\def\FH{Fr\'echet--Hoeffding\xspace}

\def\ud{\mathrm{d}}

\def\1{1}

\DeclareMathAccent{\what}{\mathord}{largesymbols}{"62}

\DeclareFontFamily{U}{mathx}{\hyphenchar\font45}
\DeclareFontShape{U}{mathx}{m}{n}{
      <5> <6> <7> <8> <9> <10>
      <10.95> <12> <14.4> <17.28> <20.74> <24.88>
      mathx10
      }{}
\DeclareSymbolFont{mathx}{U}{mathx}{m}{n}
\DeclareFontSubstitution{U}{mathx}{m}{n}
\DeclareMathAccent{\widecheck}{0}{mathx}{"71}

\begin{document}

\title{Marginal and dependence uncertainty: bounds,\newline optimal transport, and sharpness}

\author[a,1]{Daniel Bartl}
\author[a,2]{Michael Kupper}
\author[b,3]{Thibaut Lux}
\author[c,4]{Antonis Papapantoleon}
\author[a,5]{\vspace{.35em}\newline \phantom{a}\hfill with an appendix by Stephan Eckstein}
\address[a]{{Department of Mathematics, University of Konstanz, Universit\"atsstra{\ss }e 10, 78464 Konstanz, Germany}}
\address[b]{{Department of Finance, Vrije Universiteit Brussel, Pleinlaan 2, 1050 Brussels, Belgium}}
\address[c]{{Department of Mathematics, National Technical University of Athens, Zografou Campus, 15780 Athens, Greece }}

\eMail[1]{daniel.bartl@uni-konstanz.de}
\eMail[2]{kupper@uni-konstanz.de}
\eMail[3]{tlux@consult-lux.de}
\eMail[4]{papapan@math.ntua.gr}
\eMail[5]{stephan.eckstein@uni-konstanz.de}

\abstract{
Motivated by applications in model-free finance and quantitative risk management, we consider Fr\'echet classes of multivariate distribution functions where additional information on the joint distribution is assumed, while uncertainty in the marginals is also possible.
We derive optimal transport duality results for these Fr\'echet classes that extend previous results in the related literature.
These proofs are based on representation results for increasing convex functionals and the explicit computation of the conjugates.
We show that the dual transport problem admits an explicit solution for the function $f=1_B$, where $B$ is a rectangular subset of $\R^d$, and provide an intuitive geometric interpretation of this result.
The improved \FH bounds provide ad-hoc upper bounds for these Fr\'echet classes.
We show that the improved \FH bounds are pointwise sharp for these classes in the presence of uncertainty in the marginals, while a counterexample yields that they are not pointwise sharp in the absence of uncertainty in the marginals, even in dimension 2.
The latter result sheds new light on the improved \FH bounds, since \citet{tankov} has showed that, under certain conditions, these bounds are sharp in dimension 2.
}

\keyWords{Dependence uncertainty, marginal uncertainty, Fr\'echet classes, improved \FH bounds, optimal transport duality, relaxed duality, sharpness of bounds.}

\keyAMSClassification{60E15, 49N15, 28A35}
\thanksColleagues{{A previous version was entitled ``Sharpness of improved Fr\'echet--Hoeffding bounds: an optimal transport approach''.}}

\date{} \maketitle \frenchspacing

\section{Introduction}

A celebrated result in probability theory are the \FH bounds, which provide a bound on the joint distribution function (or the copula) of a random vector in case only the marginal distributions are known.
Let $\mathcal{F}(F_1^\ast,\dots,F_d^\ast)$ denote the Fr\'echet class of cumulative distribution functions (cdf's) on $\mathbb{R}^d$ with (known) univariate marginal distributions $F_1^\ast,\dots,F_d^\ast$.
Then, the \FH bounds state that the following inequalities hold for all joint distribution functions with the given marginals, \textit{i.e.} for all $F\in\mathcal{F}(F_1^\ast,\dots,F_d^\ast)$ it holds that
\begin{align}
\label{intro3}
\Big(\sum_{i=1}^d F^\ast_i(x_i) - (d-1)\Big)^+ \leq F(x_1,\dots,x_d) \leq \min_{i=1,\dots,d} F^\ast_i(x_i), 
\end{align}
for all $(x_1,\dots,x_d)\in\mathbb{R}^d$.
These bounds can be combined with results on stochastic dominance, see \textit{e.g.} \citet{mueller}, which state that for certain classes of (cost or payoff) functions $\varphi$ the inequalities on the distribution function are preserved when considering integrals of the form $\int \varphi\,\ud F$. 
In other words, the \FH bounds combined with results from stochastic dominance yield upper and lower bounds on integrals of the form $\int \varphi\,\ud F$, over all possible joint distribution functions with given marginals $F_1^\ast,\dots,F_d^\ast$ (or, equivalently, over all copulas).

These results have found several applications in financial and insurance mathematics, since they allow to derive bounds on the prices of multi-asset options and on risk measures, such as the Value-at-Risk, in the framework of \textit{dependence uncertainty}, \textit{i.e.} when the marginal distributions are known and the joint distribution is not known, see \textit{e.g.} \citet{chen,cherubini,Dhaene_etal_2002_a,Dhaene_etal_2002_b,embrechts,embrechts2} and \citet{puccetti2012b}.
Analogous results have been also derived using methods from linear programming or optimal transport theory, see \textit{e.g.} \citet{dAspremont_ElGhaoui_2006,boyle,carlier2003,Han} and \citet{Hobson_Laurence_Wang_2005_2,Hobson_Laurence_Wang_2005_1}.
The optimal transport duality, also known as pricing-hedging duality in the mathematical finance literature, states, for example, that 
\begin{align}
\sup_{F\in\mathcal{F}(F_1^\ast,\dots,F_d^\ast)} \int \varphi \, \ud F
	\ = \inf_{\varphi_1+\dots+\varphi_d\ge \varphi} \Big\{ \int \varphi_1 \ud F_1^* + \dots + \int \varphi_d \ud F_d^* \Big\}.
\end{align}
In other words, using the language of mathematical finance, the model-free upper bound on the price of an option with payoff function $\varphi$ over all joint distributions with given marginals equals the infimum over all hedging strategies which consist of investing according to $\varphi_i$ in the asset with marginal $F_i^*$, subject to the condition that $\varphi_1+\dots+\varphi_d$ dominates the payoff function $\varphi$.

The main pitfall with the framework of dependence uncertainty is that the resulting bounds are too wide to be informative for applications; \textit{e.g.} the bounds for multi-asset option prices may coincide with the trivial no-arbitrage bounds.
On the other hand, we can infer from financial and insurance markets partial information on the dependence structure of a random vector which is not utilized in the \FH bounds and more generally the framework of dependence uncertainty, where only information on the marginals is taken into account. 

These considerations have led to increased attention on frameworks that could be termed \textit{partial dependence uncertainty}, \textit{i.e.} when additional information is available on the dependence structure.
The additional information available can take several forms, for example, some authors assume that the joint distribution function is known on some subset of its domain, others assume that the correlation (or more generally a measure of association) is known,  others assume that the variance of the sum is known or bounded, and so forth. 
We refer the reader to \citet{MR3752340,Rueschendorf_2018} for an overview of this literature, with emphasis on applications to Value-at-Risk bounds.

Analogously to the \FH bounds in the framework of dependence uncertainty, several authors have developed \textit{improved} \FH bounds that correspond to the framework of partial dependence uncertainty, see \textit{e.g.} \citet{nelsen,tankov} and \citet{lux2016,Lux_Papapantoleon_2016}.
The improved \FH bounds can accommodate different types of additional information, such as the knowledge of the distribution function on a subset of the domain and the knowledge of a measure of association.
In this article, we consider the following Fr\'echet class under additional information
\begin{equation}\label{sharp2}
\mathcal{F}^{S,\pi}(F_1^\ast,\dots,F_d^\ast) 
:= \big\{F\in\mathcal{F}(F_1^\ast,\dots,F_d^\ast)\colon F(s) 
= \pi_s \text{ for all } s\in S\big\},
\end{equation}
where $S\subset\mathbb{R}^d$ is an arbitrary set and $(\pi_s)_{s\in S}$ a family with values in $[0,1]$.
In other words, we consider all joint distribution functions with known marginals $F_1^\ast,\dots,F_d^\ast$ and known value $\pi_s$ for each $s\in S$, where $S$ is a subset of the domain.
The additional information on the joint distribution assumed in this class may not be directly observable in the markets, but can be implied from multi-asset option prices or other derivatives by arguments analogous to \citet{breeden}.
Improved \FH bounds for this class have been derived in \cite{tankov,lux2016} and read as follows:
\begin{align}\label{sharp3}
\begin{split}
&\Big(\sum_{i=1}^d F^\ast_i(x_i) - (d-1)\Big)^+ \vee \max\Big\{ \pi_s - \sum_{i=1}^d \big(F_i^\ast(s_i) - F_i^\ast(x_i)\big)^+ : s\in S\Big\} \\
&\quad	\le F(x_1,\dots,x_d) \le 
\min_{i=1,\dots,d} F^\ast_i(x_i) \wedge \min\Big\{ \pi_s + \sum_{i=1}^d \big(F_i^\ast(x_i) - F_i^\ast(s_i)\big)^+ : s\in S\Big\}.	
\end{split}
\end{align}
The authors in \cite{tankov,lux2016} have also used the improved \FH bounds in order to derive bounds for the prices of multi-asset derivatives in a framework of partial dependence uncertainty, and showed that the additional information incorporated in the bounds leads to a notable tightening of the option price bounds relative to the case without additional information.

One could ask though whether this framework is realistic for applications, in particular whether the assumption of perfect knowledge of the marginal distributions is supported by empirical evidence or stems from mathematical convenience.
We take the view that perfect knowledge of the marginals is not a realistic assumption, and are thus interested in frameworks that combine uncertainty in the marginals with partial uncertainty in the dependence structure. 
Toward this end, we introduce two Fr\'echet classes that correspond to this framework, and we are interested in studying their properties.

The classes we introduce allow us to consider simultaneously uncertainty in the marginal distributions, measured either by 0-th or by first order stochastic dominance, and additional information on the dependence structure, provided by values $\pi_s$ for $s\in S$.
Let us thus consider the following relaxed version of the class $\mathcal F^{S,\pi}$ in \eqref{sharp2}, provided by
\begin{align}\label{eq:fclass-0} 
	\mathcal{F}^{S,\pi}_{\preceq_0}(F_1^\ast,\dots,F_d^\ast)
		:= \bigg\{ cF : \begin{array}{l}
			c\in[0,1] \text{ and } F \text{ is a cdf on } \mathbb{R}^d \text{ such that } cF_i\preceq_0 F_i^\ast \\
			\text{for all } i=1,\dots, d \text{ and } cF(s)\leq \pi_s \text{ for all }s\in S  
		\end{array} \bigg\}, 
\end{align}
where $F_i$ is the $i$-th marginal distribution of $F$, and $cF_i\preceq_0 F_i^\ast$ means that $F_i^\ast$ dominates $cF_i$ in the $0$-th stochastic order, \textit{i.e.}~$cF_i\preceq_0 F_i^\ast$ if and only if $cF_i(t)-cF_i(s)\leq F_i^\ast(t)-F_i^\ast(s)$ for all $s\le t$.\footnote{Denote by $\mu$ and $\mu^\ast$ the (sub-)probabilities on the real line associated to $cF_j$ and $F_j^\ast$. Then a Dynkin argument shows that $cF_j\preceq_0 F_j^\ast$ if and only if $\mu(B)\leq \mu^\ast(B)$ for every Borel subset of $\mathbb{R}$.} 
(Note that for $c=1$, it follows from $F_i\preceq_0 F_i^\ast$ that $F_i=F_i^\ast$.) 
In other words, we consider the class of joint distribution functions (associated with sub-probability measures) whose marginals are dominated by $F_1^\ast,\dots,F_d^\ast$ in the 0-th stochastic order and whose value is smaller than $\pi_s$ for each point $s$ in a subset $S$ of the domain.
Moreover, we consider another relaxed version of the class $\mathcal F^{S,\pi}$ in \eqref{sharp2}, provided by
\begin{align}\label{eq:fclass-1}
	\mathcal{F}^{S,\pi}_{\preceq_1}(F_1^\ast,\dots,F_d^\ast)
		:= \bigg\{ F : \begin{array}{l}
			F \text{ is a cdf on } \mathbb{R}^d \text{ such that } F_i\preceq_1 F_i^\ast \text{ for all } i=1,\dots, d \\
			\text{ and } F(s)\leq \pi_s \text{ for all }s\in S  
		\end{array} \bigg\}, 
\end{align}
where $F_i\preceq_1 F_i^\ast$ means that $F_i$ first-order stochastically dominates $F^\ast_i$.
This class is very similar to the previous one, however now we consider probability measures on $\R^d$. 
Let us mention that there exist in the literature tests of first order stochastic dominance, see \textit{e.g.} \citet{Schmid_Trede_1996}.

These two classes belong to the framework described above, \textit{i.e.} they allow us to take into account and combine uncertainty in the marginal distributions with additional partial information on the dependence structure.
An easy computation using arguments from copula theory, that is deferred to Appendix \ref{sec:derivation}, shows that the improved \FH bounds of \cite{lux2016,tankov} hold true also for these two classes.
An analogous result appears already in \citet{puccetti2016}.
The contributions of this paper are then threefold:
\begin{itemize}[leftmargin=*]
\item We provide optimal transport, or pricing-hedging, duality results for each of the Fr\'echet classes 
	  $\mathcal{F}^{S,\pi},\mathcal{F}^{S,\pi}_{\preceq_0}$ and $\mathcal{F}^{S,\pi}_{\preceq_1}$. 
	  In other words, we show that the optimal transport duality holds in the presence of additional information; this generalizes previous results in the related literature, see \textit{e.g.} \citet{Rachev_Rueschendorf_1994}.
	  In the context of mathematical finance, we show that the pricing-hedging duality results hold also in the presence of additional information, in which case the hedging portfolio should also consist of positions in multi-asset options for which the additional information is available.
	  Moreover, the uncertainty in the marginals translates into trading constraints on the `dual' side, such as short-selling constraints.

\item We show that the dual transport problem for the function $f=1_B$, for rectangular sets $B\subset \R^d$, admits an explicit solution in the class 
	  $\mathcal{F}_{\preceq_0}^{S,\pi}$.
	  In the language of mathematical finance, we show that the superhedging problem for a multi-asset digital option under short-selling constraints admits an explicit solution, and explain the intuition behind this result.

\item Finally, we discuss the pointwise sharpness of the upper improved \FH bound for each of the classes 
	  $\mathcal{F}^{S,\pi},\mathcal{F}^{S,\pi}_{\preceq_0}$ and $\mathcal{F}^{S,\pi}_{\preceq_1}$.
	  An upper bound $\mathcal B$ is called \textit{sharp} for a certain class $\mathcal C$ if $\mathcal B\in\mathcal C$ and is called \textit{pointwise sharp} for $\mathcal C$ if
	  \[\sup_{F\in\mathcal C} F(x) = \mathcal B(x) \ \text{ for all $x\in\dom(F)$}.\] 
	  More specifically, we show that the upper improved \FH bound is pointwise sharp for the classes $\mathcal{F}^{S,\pi}_{\preceq_0}$ and $\mathcal{F}^{S,\pi}_{\preceq_1}$.
	  In addition, by means of a counterexample, we show that the same bound is \textit{not} pointwise sharp for the class $\mathcal{F}^{S,\pi}$, even in dimension $d=2$.
	  The latter result is surprising since \citet{tankov} has showed that under certain conditions on the set $S$ the upper bound is not only pointwise sharp but even sharp, \textit{i.e.} he actually showed that the upper improved \FH bound belongs to the Fr\'echet class $\mathcal{F}$.
\end{itemize}

This article is structured as follows: in Section \ref{sec:duality} we derive optimal transport dualities for the three Fr\'echet classes introduced above. 
In Section \ref{sec:solution.box} we provide an explicit solution of the dual transport problem for the function $f=1_B$, in the class $\mathcal{F}_{\preceq_0}^{S,\pi}$.
In Section \ref{sec:proof.main.sharp} we present and discuss the pointwise sharpness results for the improved \FH bounds.
Finally, Appendix \ref{sec:counterexample.not.sharp} contains the aforementioned counterexample, while Appendix \ref{sec:derivation} contains the derivation of the improved \FH bounds for the classes $\mathcal{F}^{S,\pi}_{\preceq_0}$ and $\mathcal{F}^{S,\pi}_{\preceq_1}$.

\section{Transport and relaxed transport duality under additional information}
\label{sec:duality}

In this section, we establish our main duality results. 
We derive a dual representation for the transport problem of maximizing the expectation of a $d$-dimensional cost or payoff function over probability or sub-probability measures whose univariate marginals are either given or dominated in the 0-th stochastic order and whose mass is prescribed on certain rectangles in $\mathbb{R}^d$. 
Using similar arguments, we also obtain a dual representation for a transport problem involving constraints in the form of estimates on the marginal distributions in the first stochastic order. 
Moreover, as a corollary we establish duality for a transport problem with constraints on the maximum of random variables.
We follow the notation of optimal transport theory in this section, and thus work with measures instead of distribution functions.

We start by introducing useful notions and notation. 
Let us denote by $ca^+(\R^d)$ the set of all finite measures on the Borel $\sigma$-field of $\mathbb{R}^d$, $d\ge1$, and by $ca^+_1(\R^d)$ (resp.~$ca^+_{\leq 1}(\R^d)$) the subset of those measures $\mu$ satisfying $\mu(\mathbb{R}^d)=1$ (resp.~$\mu(\mathbb{R}^d)\leq 1$). 
The space $\R^d$ might also be omitted from the notation in case there is no ambiguity.

Let $\nu_1,\dots,\nu_d\in ca^+_1(\mathbb{R})$ and define the sets
\[
A^i:=(-\infty,A^i_1]\times\cdots\times(-\infty,A^i_d]\subset\mathbb{R}^d
\]
for $i\in I$, where $I$ is an arbitrary index set.
Define for any \textit{cost} or \textit{payoff} function $f\colon\mathbb{R}^d\to\mathbb{R}$ the set
\[ 
\Theta(f):=\Big\{ (f_1,\dots,f_d,a) : f_1(x_1)+\dots+f_d(x_d) + \sum_{i\in I} a^i\1_{A^i}(x)\geq f(x),
\ \text{for all } x\in\mathbb{R}^d  \Big\}, 
\]
where $f_i\colon\mathbb{R}\to\mathbb{R}$ are bounded measurable functions and $a=(a^i)\in\mathbb{R}^I$, such that $a^i=0$ for all but finitely many $i\in I$.
Moreover, let $0\leq\underline{\pi}^i\leq\overline{\pi}^i\leq 1$ and define
\[ 
\pi(f_1,\dots,f_d,a) := \int_{\mathbb{R}}f_1\,\ud\nu_1 + \cdots+\int_{\mathbb{R}}f_d\,\ud\nu_d
	+ \sum_{i\in I}\big( a^{i+}\overline{\pi}^i-a^{i-}\underline{\pi}^i \big),
\]
for every $(f_1,\dots,f_d,a)\in \Theta(f)$, where $a^{i+}$ and $a^{i-}$ denote the positive and negative part of $a^i$ respectively, \textit{i.e.} $a^{i+}=\max\{a^i,0\}$ and $a^{i-}=\max\{-a^i,0\}$.
Denote by $\Theta_0(f)$ the set of all $(f_1,\dots,f_d,a)\in\Theta(f)$ such that
$f_1,\dots,f_d\geq 0$ and $a^i\geq 0$ for all $i\in I$.
Now define the functionals
\[ 
\phi(f) := \inf\big\{ \pi(f_1,\dots,f_d,a) : (f_1,\dots,f_d,a)\in\Theta(f) \big\} 
\]
and
\[ 
\phi_0(f) := \inf\big\{ \pi(f_1,\dots,f_d,a) : (f_1,\dots,f_d,a)\in\Theta_0(f) \big\}.
\]
Moreover, consider the sets of measures 
\[
\mathcal{Q} := \Big\{ \mu\in ca_1^+(\mathbb{R}^d):\mu_1=\nu_1,\,\dots,\,\mu_d=\nu_d 
	\text{ and } \underline{\pi}^i\leq\mu(A^i)\leq \overline{\pi}^i, \text{ for all } i\in I \Big\}
\]
and
\[
\mathcal{Q}_0 := \Big\{ \mu\in ca_{\leq 1}^+(\mathbb{R}^d):\mu_1\preceq_0 \nu_1,\,\dots,\,\mu_d\preceq_0 \nu_d 
	\text{ and } \mu(A^i)\leq \overline{\pi}^i, \text{ for all } i\in I \Big\},
\]
where $\mu_j$ denotes the $j$-th marginal of the measure $\mu$, while $\mu_j\preceq_0\nu_j$ should be understood as $\mu_j(B)\leq\nu_j(B)$ for every Borel set $B\subset\mathbb{R}$; the latter condition is also known as 0-th order stochastic dominance.

The following theorem establishes a Monge--Kantorovich duality under additional constraints in the context of optimal transportation or a pricing-hedging duality under additional information in the context of mathematical finance.
Indeed, in the context of optimal transportation we seek to maximize the total cost $\int f\ud \mu$ relative to transport plans $\mu$ with marginals $\nu_1,\dots,\nu_d$ which in addition satisfy the constraint $\underline{\pi}^i\leq\mu(A^i)\leq \overline{\pi}^i$ for $i\in I$.
We also consider a relaxed version of this problem, where we seek to maximize the same total cost relative to transport plans that are dominated by $\nu_1,\dots,\nu_d$ and satisfy the additional constraint $\mu(A^i)\leq \overline{\pi}^i$ for $i\in I$.

In the context of mathematical finance, let $f$ denote the payoff function of an option depending on multiple assets, whose joint distribution is $\mu$. 
Then, the right hand side in \eqref{eq:duality-ii} is the model-free superhedging price for this option assuming that the marginal distributions are known (\textit{i.e.} $\mu_1=\nu_1,\dots,\mu_d=\nu_d$), while there is also additional information present, in the form of the bounds $\underline\pi^i,\overline\pi^i$ on the price of the multi-asset digital options $1_{A^i}$, $i\in I$. 
The left hand side in \eqref{eq:duality-ii} describes hedging or trading strategies which consist of investing in options with payoff $f_i$ in the $i$-th marginal $\nu_i$, $i\in\{1,\dots,d\}$, and also buying $a^{i+}$ and selling $a^{i-}$ digital options with payoff $1_{A^i}$ at the prices $\overline\pi^i$ and $\underline\pi^i$ respectively, $i\in I$, subject to the requirement that the sum of these payoffs dominate $f$. 
The duality in \eqref{eq:duality-i} represents a relaxation of the problem described above, where on the one side uncertainty in the marginals is taken into account, while on the other side the trading strategies are subject to short-selling constraints (\textit{i.e.} they are positive).

\begin{definition}
A \textit{trading strategy} $(f_1,\dots,f_d,a)$ which satisfies
\[
(f_1,\dots,f_d,a)\in\Theta(\varepsilon) \quad \mbox{and} \quad \pi(f_1,\dots,f_d,a)\le 0
\]
for some $\varepsilon>0$ is called a \emph{uniform strong arbitrage}.
\end{definition}

\begin{remark}
The strategy described above is called a uniform strong arbitrage because its price at inception is less than or equal to zero, while its outcome is bounded from
below by $\varepsilon>0$.
The next theorem relates the absence of uniform strong arbitrage with the existence of an element in $\mathcal Q$, the set of probability measures with given marginals that satisfy the condition $\underline{\pi}^i\leq\mu(A^i)\leq \overline{\pi}^i$ for all $i\in \mathcal I$.
In other words, the absence of arbitrage allows us to conclude something about probability measures with given marginals, and \textit{vice versa}.
Notice that the absence of uniform strong arbitrage is a very weak condition, that is implied by the classical no-arbitrage conditions.
\end{remark}

\begin{theorem}
\label{thm:main.rep}
Let $f\colon\mathbb{R}^d\to\mathbb{R}$ be an upper semicontinuous and bounded function.
Then, there does not exist a uniform strong arbitrage if and only if $\mathcal{Q}$ is not empty. 
In this case,
\begin{align}\label{eq:duality-ii}
\phi(f)=\max_{\mu\in\mathcal{Q}} \int_{\mathbb{R}^d}f \,\ud\mu.
\end{align}
Moreover, 
\begin{align}\label{eq:duality-i}
\phi_0(f)=\max_{\mu\in\mathcal{Q}_0} \int_{\mathbb{R}^d}f \,\ud\mu.
\end{align}
\end{theorem}

\begin{remark}
The optimal transport duality under additional information \eqref{eq:duality-ii} appears in a similar form in \citet[Theorem~3.2]{Lux_Rueschendorf_2017}. 
These two results were developed in parallel, however their proofs are completely different. 
Moreover, in \cite{Lux_Rueschendorf_2017} the authors consider the Fr\'echet class of $d$-dimensional probability distributions with given marginals, whose copulas are bounded from below and above by arbitrary quasi-copulas; a quasi-copula generalizes the notion of a copula. 
In view of \eqref{eq:duality-ii}, our formulation is slightly more general as we do not require the bounds  $(\underline{\pi}^i, \overline{\pi}^i)_{i\in I}$ to have a particular structure as imposed by a quasi-copula.
\end{remark}

The proof of this theorem builds on the following representation result for convex and increasing functions, and an explicit computation of the conjugates.
Let us first introduce the functionals $\phi^\ast_{C_b}$ and $\phi^\ast_{U_b}$ which are defined as follows:
\begin{align}\label{eq:phi^*-functionals}
\phi^\ast_{C_b}(\mu):=\sup_{f\in C_b} \Big\{\int f\ud\mu-\phi(f)\Big\}
	\ \ \text{ and } \ \
\phi^\ast_{U_b}(\mu):=\sup_{f\in U_b} \Big\{\int f\ud\mu -\phi(f)\Big\},
\end{align}
where $U_b$ denotes the set of all bounded upper semicontinuous functions $f:\R^d\to\R$, and $C_b$ the set of all bounded continuous functions.
Then the following holds true:

Let $\phi\colon U_b\to\mathbb{R}$ be a convex and increasing function, and assume that for every sequence $(f^n)$ of continuous bounded functions such that $f^n$ decreases pointwise to 0, it holds that $\phi(f^n)\downarrow \phi(0)$. 
Then, $\phi$ admits the following representation:
\begin{align}\label{eq:limmed-entropic}
\phi(f)=\max_{\mu\in ca^+} \Big\{ \int f\ud\mu-\phi^\ast_{C_b}(\mu) \Big\}	
\end{align}
for all $f\in C_b$.
Assume, in addition, that $\phi^\ast_{C_b}(\mu)=\phi^\ast_{U_b}(\mu)$ for any 
$\mu\in ca^+$, then $\phi(f)=\max_{\mu\in ca^+} \{ \int f\ud\mu-\phi^\ast_{U_b}(\mu) \}$ for all $f\in U_b$.
The proof is similar to \citet[Theorem 2.2]{bartl2017duality}; see also \citet[Theorem A.5]{RobHeding}.

\begin{proof}[of Theorem \ref{thm:main.rep}]
We start by noting that
\[ 
\Theta(\lambda f)=\lambda \Theta(f)
	\quad\text{and}\quad
\Theta(f)+\Theta(g)\subset\Theta(f+g) 
\]
for every $\lambda>0$ and every two functions $f,g\colon\mathbb{R}^d\to\mathbb{R}$.
Moreover, it holds that
\[
\pi(\lambda f_1,\dots,\lambda f_d,\lambda a) = \lambda\pi(f_1,\dots,f_d,a),
\]
while from the inequalities $a^++b^+\ge(a+b)^+$ and $a^-+b^-\ge(a+b)^-$ and the condition $0\le \underline\pi^i\le \overline\pi^i$ it follows that
\[ 
\pi(f_1+g_1,\dots,f_d+g_d,a+b) \leq \pi(f_1,\dots,f_d,a) + \pi(g_1,\dots,g_d,b).
\]
Therefore, we get that $\phi$ is a sublinear functional, \textit{i.e.} $\phi(\lambda f)=\lambda \phi(f)$ for $\lambda>0$ and	$\phi(f+g)\leq\phi(f)+\phi(g)$.
The same arguments apply to $\Theta_0$, hence $\phi_0$ is sublinear as well.
Moreover, since
\[ 
(m,0,\dots,0)\in\Theta_0(m) \quad\text{and}\quad \pi(m,0,\dots,0)=m \quad\text{for } m\in\R_+,
\]
it follows that $\phi(m)\leq\phi_0(m)\leq m$. 

The main part of the proof is to show that if $\phi$ is real-valued, then $\phi$ has the representation \eqref{eq:duality-ii}.

\textit{Step 1:} 
We claim that if there does not exist uniform strong arbitrage, then $\phi(m)=m$ for all $m\in\mathbb{R}$.
We have already shown that $\phi(m)\leq m$.
On the other hand, if $\phi(m)<m-\varepsilon$ for some $\varepsilon>0$, there exists $(f_1,\dots,f_d,a)\in\Theta(m)$ such that $\pi(f_1,\dots,f_d,a)\le m-\varepsilon$.
Define 
\[
f_j'(x):= f_j(x)-\frac{m-\varepsilon}{d}
\]
for every $x\in\R$ and $j=1,\dots,d$.
Then
\[
(f_1',\dots,f_d',a)\in\Theta(\varepsilon) \quad\text{but}\quad \pi(f_1',\dots,f_d',a)\le 0, 
\]
which contradicts the assumption that there does not exist uniform strong arbitrage.

\textit{Step 2:}
We claim that $\phi$ and $\phi_0$ are continuous from above on $C_b$, \textit{i.e.} that $\phi(f^n)\downarrow 0$ for every sequence $(f^n)$ in $C_b$ such that $f^n\downarrow 0$ pointwise.
Let us fix such a sequence $(f^n)$, some $\varepsilon>0$, and let $l$ be such that
\begin{align}\label{eq:nu-bound}
\nu_i([-l,l]^c)\leq \frac{\varepsilon}{d\sup_x |f^1(x)|},
	\quad \text{for all } i\in\{1,\dots,d\}.
\end{align}
The set $K:=[-l,l]^d\subset\mathbb{R}^d$
is compact, therefore we can apply Dini's Lemma to obtain some index $n_0$ such that 
\begin{align}\label{eq:bound-K}
f^n\1_K\leq \varepsilon \quad\text{for all }n\geq n_0.
\end{align}
Now we define 
$f_j(x):=\sup_x |f^1(x)|\1_{[-l,l]^c}$ 	 
for $j\in\{1,\dots,d\}$ such that 
\[
f^n \1_{K^c}(x)
	\leq f^1 \1_{K^c}(x)
	\leq f_1(x_1)+\dots+f_d(x_d),
\]
thus $(f_1,\dots,f_d,0)\in\Theta_0(f^n\1_{K^c})$. 
Therefore we have
\begin{align}\label{eq:bound-Kc}
\phi_0(f^n\1_{K^c})
	\leq \pi(f_1,\dots,f_d,0)
	= \int_{\mathbb{R}}f_1\,\ud\nu_1 +\cdots+ \int_{\mathbb{R}}f_d\,\ud\nu_d
	\leq \varepsilon,
\end{align}
where the last inequality follows from \eqref{eq:nu-bound} and the definition of $f_j$.
Subadditivty together with \eqref{eq:bound-K} and \eqref{eq:bound-Kc} thus implies 
\[
\phi_0(f^n) \leq \phi_0(f^n\1_K)+\phi_0(f^n\1_{K^c}) \leq \varepsilon+\varepsilon\]
for all $n\geq n_0$.
As $\varepsilon$ was arbitrary, we conclude that indeed $\phi_0(f^n)\downarrow 0=\phi_0(0)$ and thus also $\phi(f^n)\downarrow0$, since  $\phi(f^n)\leq\phi_0(f^n)$ and $\phi(0)=0$.

\textit{Step 3:}
In a final step we want to show that
\[ 
\phi^*_{C_b}(\mu) = \phi^*_{U_b}(\mu) 
	= \begin{cases}
		0, &\text{if } \mu\in\mathcal{Q},\\
		+\infty, &\text{ otherwise},
	\end{cases}
\]
and
\[ 
\phi_{0,C_b}^*(\mu) = \phi_{0,U_b}^*(\mu)
	=\begin{cases}
		0, &\text{if } \mu\in\mathcal{Q}_0,\\
	+\infty, &\text{ otherwise}
	\end{cases}
\]
for every $\mu\in ca^+$. 
Here $\phi_{0,C_b}^*$ and $\phi_{0,U_b}^*$ are defined analogously to  $\phi^*_{C_b}$ and $\phi^*_{U_b}$; cf. \eqref{eq:phi^*-functionals}.

On the one hand, we will show that the conjugates take the value $+\infty$ whenever the measure $\mu \notin \mathcal Q$, resp. $\mu \notin \mathcal Q_0$.
Notice that, by definition, $0\leq\phi^*_{C_b}\leq\phi^*_{U_b}$ and similarly $0\leq\phi_{0,C_b}^*\leq\phi_{0,U_b}^*$.
Since $\phi(m)=m$ for $m\in\mathbb{R}$, it follows that
\[
\phi^\ast_{C_b}(\mu)
	\geq \sup_{m\in\mathbb{R}} \big\{ m\mu(\mathbb{R}^d) -m \big\}
	= +\infty
\]
whenever $\mu$ is not a probability measure.
Analogously, we get that
\[ 
\phi_{0,C_b}^\ast(\mu)
	\geq \sup_{m\geq0 } \big\{ m\mu(\mathbb{R}^d) -m \big\}
	= +\infty 
\]
whenever $\mu(\mathbb{R}^d)>1$.

Let now $\mu_j(B) \neq \nu_j(B)$ (resp. $\mu_j(B)>\nu_j(B)$) for some Borel set $B\subset\mathbb{R}$ and some $j\in\{1,\dots,d\}$.
Then there exists a continuous bounded function $h\colon\mathbb{R}\to\mathbb{R}$ (resp. $h\colon\mathbb{R}\to[0,+\infty)$) such that
\[ 
\int_{\mathbb{R}}h\,\ud\mu_j> \int_{\mathbb{R}} h\,\ud\nu_j.
\]
Moreover, we can define a function $f$ via
$f(x):= h(x_j) \text{ for } x\in\mathbb{R}^d$, 
which is continuous and bounded.
We can also define 
\[
f_j(x):=h(x) \, \text{ and } \, f_k(x):=0 \text{ for } x\in\R, \,\text{ and } a=0\]
for all $k\in\{1,\dots,d\}\setminus\{j\}$.
Then, by construction it holds
\[
(f_1,\dots,f_d,a)\in\Theta(f)
	\quad \text{and} \quad
\pi(f_1,\dots,f_d,a)=\int_{\mathbb{R}}h\,\ud\nu_j \]
(resp.~$(f_1,\dots,f_d,a)\in\Theta_0(f)$),
hence we get that
\[
\int_{\mathbb{R}^d}f \ud\mu -\phi(f)
	\geq  \int_{\mathbb{R}}h\,\ud\mu_j -\int_{\mathbb{R}}h\,\ud\nu_j
	>0.
\]
Since $\phi(\lambda f)=\lambda\phi(f)$ for every $\lambda>0$, it follows that
\[ 
\phi^\ast_{C_b}(\mu)
	\geq \sup_{\lambda>0} \Big\{ \int_{\mathbb{R}^d}\lambda f \ud\mu -\phi(\lambda f) \Big\}
	=+\infty. 
\]
The same arguments show that $\phi_{0,C_b}^\ast(\mu)=+\infty$.

The final condition for $\mu\in\mathcal Q$ (resp. $\mu\in\mathcal Q_0$) reads as $\underline{\pi}^i\leq \mu(A^i)\leq \overline{\pi}^i$ (resp. $\mu(A^i)\leq\overline{\pi}^i)$, for all $i\in I$.
Assume that $\mu(A^i)\geq\overline{\pi}^i+\varepsilon$ for some $\varepsilon>0$ and some $i\in I$.
We may choose $\delta>0$ such that
\[
\nu_j((A_j^i,A_j^i+\delta)) < \frac\varepsilon{d}
\] 
for all $j\in\{1,\dots,d\}$.
Define the set
\[
C:=(-\infty,A_1^i+\delta)\times\cdots\times (-\infty,A_d^i+\delta).
\] 
Then $A^i$ and $C^c$ are closed and disjoint sets, so that Urysohn's Lemma guarantees the existence of a continuous function $f\colon\mathbb{R}^d\to\mathbb{R}$ such that 
\begin{align}\label{eq:indicatorsAC}
\1_{A^i}\leq f\leq \1_{C}.
\end{align}
Observe that
\[
f(x) \le \1_C 
	\leq \1_{A^i}(x) + \1_{(A^i_1,A^i_1+\delta)}(x_1) + \cdots + \1_{(A^i_d,A^i_d+\delta)}(x_d),
\] 
since every element in $C$ belongs either to $A^i$ or one of the sets $(A^i_j,A^i_j+\delta)$, $j\in\{1,\dots,d\}$.
Hence, it follows that
\[ 
(f_1,\dots,f_d,a)\in\Theta_0(f)\subset\Theta(f), 
\]
where we have defined
\[ 
f_j(x):=\1_{(A^i_j,A^i_j+\delta)}(x), x\in\R, \quad \text{and} \quad a^i=1, \, a^k=0\]
for all $j\in\{1,\dots,d\}$ and $k\in I\setminus\{i\}$.
Thus we have now
\begin{align}\label{eq:intrmd}
\phi(f) \leq \phi_0(f) \leq \pi(f_1,\dots,f_d,a)
	= \overline{\pi}^i+\sum_{j=1}^\ud\nu_j((A^i_+1,A^i_d+\delta))
	< \overline{\pi}^i+\varepsilon.
\end{align}
Therefore, using \eqref{eq:indicatorsAC}, \eqref{eq:intrmd} and the assumption, we get 
\[ 
\int_{\mathbb{R}^d} f\,\ud\mu-\phi_0(f) > \mu(A^i) - (\overline\pi^i + \varepsilon) \ge 0
	\quad\text{and also }\quad
\int_{\mathbb{R}^d} f\,\ud\mu-\phi(f)>0.
\] 
Therefore, a scaling argument as before shows that $\phi_{0,C_b}^\ast(\mu)=\phi^\ast_{C_b}(\mu)=+\infty$.

Last, assume that $\mu(A^i)<\underline{\pi}^i$ for some $i\in I$.
By the closedness of the set $A^i$, there exists a sequence of continuous functions $f^n$ such that 
\[
-1 \leq f^n \leq -\1_{A^i} \quad\text{and}\quad f^n\uparrow -\1_{A^i}.
\] 
Then $(0,\dots,0,a)\in\Theta(f^n)$ for $a^i=-1$ and $a^k=0$ for $k\in I\setminus\{i\}$, so that it holds 
\[
\phi(f^n)\leq \pi(0,\dots,0,a)=\underline{\pi}^i.
\]
By the dominated convergence theorem there exists an $n$ such that $\int_{\mathbb{R}^d} f^n\,\ud\mu >-\underline{\pi}^i$, hence
\[ 
\int_{\mathbb{R}^d} f^n\,\ud\mu  - \phi(f^n)>0.
\] 
A scaling argument shows again that $\phi^\ast_{C_b}(\mu)=+\infty$.

On the other hand, we will show that if $\mu\in\mathcal{Q}$ (resp.~$\mu\in\mathcal{Q}_0$) then it holds that $\phi^\ast_{U_b}(\mu)=0$ (resp. $\phi_{0,U_b}^\ast(\mu)=0$).
Indeed, let $f\colon\mathbb{R}^d\to\mathbb{R}$ be an upper semicontinuous and bounded function such that $(f_1,\dots,f_d,a)\in\Theta(f)$ (resp.~$(f_1,\dots,f_d,a)\in\Theta_0(f)$).
Then
\begin{align*}
\pi(f_1,\dots,f_d,a)
	&\geq \int_{\mathbb{R}^d} \big\{ f_1(x_1)+\dots+f_d(x_d)+ \sum_{i\in I} a^i\1_{A^i}(x) \big\}\,\mu(\ud x)
	\geq \int_{\mathbb{R}^d} f(x)\,\mu(\ud x).
\end{align*}
Therefore, we have that
\[ 
\phi(f)\geq  \int_{\mathbb{R}^d} f\,\ud\mu
	\quad\text{and}\quad
\phi_0(f)\geq  \int_{\mathbb{R}^d} f\,\ud\mu,
\]
which immediately yields the claim.
This concludes Step 3.

Now, in order to deduce that $\phi(f)$ and $\phi_0(f)$ have the desired representation, we will make use of representation \eqref{eq:limmed-entropic}. 
The sublinearity of $\phi$ implies in particular that it is convex.
Moreover, for $f\ge g$ it holds that $\Theta(f) \subseteq \Theta(g)$, hence $\phi(f) \ge \phi(g)$, \textit{i.e.} $\phi$ is also increasing, while the second step shows that $\phi$ satisfies the remaining condition for representation \eqref{eq:limmed-entropic} to hold.
The same arguments apply also for $\phi_0$.
Therefore, \eqref{eq:limmed-entropic} allows us to obtain
\[ 
\phi(f)
	= \max_{\mu\in ca_+} \Big\{ \int_{\mathbb{R}^d}f\,\ud\mu -\phi^\ast_{C_b}(\mu)\Big\}
	= \max_{\mu\in\mathcal{Q}} \int_{\mathbb{R}^d}f\,\ud\mu,
\]
and 
\[ 
\phi_0(f)
	= \max_{\mu\in ca_+} \Big\{ \int_{\mathbb{R}^d}f\,\ud\mu -\phi_{0,C_b}^\ast(\mu)\Big\}
	= \max_{\mu\in\mathcal{Q}_0} \int_{\mathbb{R}^d}f\,\ud\mu,  
\]
for every bounded and upper semicontinuous function $f$.
In addition, we get that the sets $\mathcal{Q}$ and $\mathcal{Q}_0$ are not empty since $\phi$ and $\phi_0$ are real-valued.

Finally, in order to conclude the proof, notice that if $\mathcal{Q}$ is not empty, then there does not exist uniform strong arbitrage. Indeed, for any $\varepsilon>0$ and $(f_1,\dots,f_d,a)\in\Theta(\varepsilon)$ it follows that for $\mu\in\mathcal Q$ it holds 
\begin{align*}
\pi(f_1,\dots,f_d,a)
	&\geq \int_{\mathbb{R}^d} \Big\{ f_1(x_1)+\dots+f_d(x_d) + \sum_{i\in I} a^i\1_{A^i}(x) \Big\} \, \mu(\ud x)\\
	&\geq \int_{\mathbb{R}^d} \varepsilon\,\mu(\ud x) = \varepsilon.
\end{align*}
\end{proof}

As a corollary of Theorem \ref{thm:main.rep} we derive in the following a duality result for a maximum transport problem. 
This problem corresponds to the situation where, besides the marginal distributions, the value of the measures is prescribed on an increasing track in $\mathbb{R}^d$. 
In terms of random variables, this is equivalent to knowing the distribution of the maximum of $d$ random variables. 

\begin{corollary}[Maximum transport problem]
Let $I=\mathbb{R}$, $A^i=(-\infty,i]^d$ and $\underline{\pi}^i=\overline{\pi}^i=\nu_{\max}((-\infty,i])$ for some measure $\nu_{\max}\in ca_1^+(\mathbb{R})$.
Then 
\begin{align}\label{eq:Q-max}
\mathcal{Q} = \Big\{ \mu\in ca_1^+(\mathbb{R}^d): \mu_1=\nu_1,\,\dots,\,\mu_d=\nu_d 
	\text{ and } \mu\circ{\max}^{-1}=\nu_{\max} \Big\},
\end{align}
and for every upper semicontinuous bounded function $f\colon\mathbb{R}^d\to\mathbb{R}$ one has
\begin{align*}
\phi(f) = \inf\Big\{ \sum_{j=1}^d \int_{\mathbb{R}} f_j\,\ud\nu_j + \int_\mathbb{R} g\,\ud\nu_{\max} : f_1,\dots,f_d,g \Big\}, 
\end{align*}
where $f_1,\dots,f_d,g\colon\mathbb{R}\to\mathbb{R}$ are bounded and measurable functions such that
\begin{align}\label{eq:ineq-max}
f_1(x_1)+\dots+f_d(x_d)+g(\max x) \geq f(x), \ \text{ for all } x\in\mathbb{R}^d,
\end{align}
where $\max x:=\max_{j=1,\dots,d} x_j$ for $x\in\mathbb{R}^d$.
\end{corollary}

\begin{proof}
Let $\mu\in\mathcal Q$.
Using that $\underline{\pi}^i=\overline{\pi}^i$, we get
\[
\mu\circ{\max}^{-1}((-\infty,i]) = \mu(A^i) = \overline{\pi}^i = \nu_{\max}((-\infty,i])
\]
for all $i\in\mathbb{R}$. 
Hence, it follows that $\mathcal{Q}$ has the form given in \eqref{eq:Q-max}.
	
Let us now define
\[
\phi_{\max} (f) = \inf\Big\{ \sum_{j=1}^d \int_{\mathbb{R}} f_j\,\ud\nu_j + \int_\mathbb{R} g\,\ud\nu_{\max} : 
	f_1,\dots,f_d,g \Big\},
\]
where $f_1,\dots,f_d,g$ satisfy inequality \eqref{eq:ineq-max}.
We want to show that $\phi(f)=\phi_{\max}(f)$.
On the one hand, notice that the right hand side is smaller than the left hand side.
Indeed, for all $(f_1,\dots,f_d,a) \in \Theta(f)$, we can define 
\[
g:=\sum_{i\in I} a^i\1_{(-\infty,i]}
\] 
such that 
\[ 
\int_{\mathbb{R}} g\,\ud\nu_{\max} = \sum_{i\in I} a^i\overline{\pi}^i
	\quad\text{and}\quad 
g(\max x) = \sum_{i\in I} a^i \1_{A^i}(x). 
\]
On the other hand, let $f_1,\dots,f_d,g$ be such that
\[
f_1(x_1)+\dots+f_d(x_d)+g(\max x) \geq f(x), \quad\text{for all }x\in\mathbb{R}^d,
\]
then, using the structure of the set $\mathcal Q$, we have
\begin{align*}
\sum_{j=1}^d \int_{\mathbb{R}} f_j\,\ud\nu_j + \int_\mathbb{R} g\,\ud\nu_{\max}
	&= \int_{\mathbb{R}^d}\big\{ f_1(x_1)+\dots+f_d(x_d)+g(\max x) \big\} \,\mu(\ud x)\\
	&\geq \int_{\mathbb{R}^d}f(x) \,\mu(\ud x).
\end{align*}
Therefore, taking the infimum and the supremum on the two sides of the inequality above, and using the conclusion from the first part, we get that
\[ 
\phi(f) \ge \phi_{\max}(f) \ge \sup_{\mu\in\mathcal{Q}}\int_{\mathbb{R}^d}f \,\ud\mu.
\]
Theorem \ref{thm:main.rep} yields now that all inequalities are actually equalities.
\end{proof}

Next, we provide another relaxation of the duality in \eqref{eq:duality-ii} which follows along the same lines of reasoning as \eqref{eq:duality-i}.
In particular, this can be interpreted again as a pricing-hedging duality, where the superhedging problem involves uncertainty both in the marginals and in the joint distribution, while the hedging strategy takes into account bid and ask prices on single-asset options and the trading of multi-asset digital options.

Let us fix $\underline{\nu}_j, \overline{\nu}_j\in ca^+_1(\mathbb{R})$ for each $j=1,\dots,d$, such that $\overline{\nu}_j$ first-order stochastically dominates $\underline{\nu}_j$. 
Recall that $\underline{\nu}_j\preceq_1\overline{\nu}_j$ in the first order stochastic dominance if $\underline{\nu}_j(-\infty,t]\ge \overline{\nu}_j(-\infty,t]$ for all $t\in\mathbb{R}$.
Let $f\colon\mathbb{R}^d\to\mathbb{R}$ be a cost or payoff function and define the set
\[
	\Theta_1 (f)
		:= \Big\{ (f_1,g_1,\dots,f_d,g_d,a): \sum_{j=1}^d \big(f_j(x_j)-g_j(x_j))+\sum_{i\in I} a^i1_{A^i}(x)\geq f(x), \forall x\in\mathbb{R}^d \Big\},
\]
where $f_j,g_j\colon\mathbb{R}\to\mathbb{R}$ are non-decreasing, bounded and continuous functions and $a^i\leq0$.
Define
\[
	\pi(f_1,g_1,\dots,f_d,g_d,a) 
		:= \sum_{j=1}^d\Big( \int_{\R}f_j\,\ud\overline{\nu}_j -\int_{\R}g_j\,\ud\underline{\nu}_j \Big) + \sum_{i\in I}a^i\underline{\pi}^i,
\]
for all $(f_1,g_1,\dots,f_d,g_d,a) \in \Theta_1(f)$ and further define the functional
\[
	\phi_1(f) := \inf \big\{ \pi(f_1,g_1,\dots,f_d,g_d,a) : (f_1,g_1,\dots,f_d,g_d,a) \in \Theta_1(f) \big\}. 
\]
Moreover, consider the set of measures
\[
	\mathcal{Q}_1 
		:= \Big\{ \mu\in ca^+_1(\mathbb{R}^d): \overline{\nu}_j\preceq_1\mu_j\preceq_1\underline{\nu}_j	
					\text{ and } \underline{\pi}_i\leq\mu(A^i) \mbox{ for all }i,j \Big\}.
\]
Then the following holds.

\begin{proposition}
Let $f\colon\mathbb{R}^d\to\mathbb{R}$ be an upper semicontinuous and bounded function.
Then, if $\phi_1(\varepsilon)>0$ for every $\varepsilon>0$, it holds
\[
	\phi_1(f) = \sup_{\mu\in\mathcal{Q}_1} \int_{\mathbb{R}^d} f\,\ud\mu. 
\]
\end{proposition}

\begin{proof} 
The proof follows along the same lines as the proof of Theorem \ref{thm:main.rep}, hence we only provide a sketch.
To start with, one can check that $\phi_1\colon U_b\to\mathbb{R}\cup\{-\infty\}$ is a sublinear and monotone functional which satisfies $\phi_1(m)\leq m$ for all $m\in\mathbb{R}$.

\textit{Step 1}. 
Analogously to Theorem \ref{thm:main.rep}, it follows that $\phi_1(m)=m$ for all $m\in\mathbb{R}$.

\textit{Step 2}. 
Recall that for two probabilities $\nu,\nu'\in ca^+_1(\mathbb{R})$, one has $\nu \preceq_1 \nu^\prime$ if and only if $\int_{\mathbb{R}} f\,\ud\nu \leq\int_{\mathbb{R}} f\,\ud\nu^\prime$ for every non-decreasing bounded continuous function $f\colon\mathbb{R}\to\mathbb{R}$, which is a straightforward application of integration by parts.
In particular, if $f_i,g_i\colon\mathbb{R}\to\mathbb{R}$ are non-decreasing bounded continuous functions which satisfy 
\[
	\sum_{j=1}^d \big(f_j(x_j)-g_j(x_j))+\sum_{i\in I} a^i1_{A^i}(x)\geq f(x)\quad \mbox{for all }x\in\mathbb{R}^d,
\]
and $\mu\in ca^+_1(\mathbb{R}^d)$ is such that $\underline{\nu}_j\preceq_1\mu_j\preceq_1\overline{\nu}_j$ for all $1\le j\le d$ and $\underline{\pi}_i\leq\mu(A^i)$ for all $i\in I$, then it holds
\begin{align*}
\int_{\mathbb{R}^d} f\,\ud\mu
	&\le \sum_{j=1}^d\int_{\mathbb{R}} (f_j-g_j)\,\ud\mu_j +\sum_{i\in I}  a^i\mu(A^i)
	\leq \sum_{j=1}^d\int_{\mathbb{R}} f_j\,\ud\overline{\nu}_i -\int_{\mathbb{R}} g_j\,\ud\underline{\nu}_i 	+\sum_{i\in I}  a^i\underline{\pi}^i.
	\end{align*}
Since $f_i,g_i$ and $a^i$ were arbitrary, it follows that $\phi_1(f)\geq \int_{\mathbb{R}^d} f\,\ud\mu$.

\textit{Step 3}.
Let $(f^n)$ be a sequence of bounded continuous functions which decreases pointwise to 0.
For $\varepsilon>0$, fix $m\in\mathbb{N}$ such that
\[ 
	\max\big\{\overline{\nu}_j([m-1,\infty)),\underline{\nu}_j((-\infty,-m+1])\big\}
		\leq\frac{\varepsilon}{2d\sup_x|f^1(x)|} 
\]
for every $j=1,\dots,d$, and define the increasing functions
\[ 
	f_j(t):=\sup_{x\in\mathbb{R}^d}|f^1(x)| \cdot \big(1+0\vee(t+1-m)\wedge 1\big)
		\ \ \text{and}\ \
	g_j(t):=\sup_{x\in\mathbb{R}^d}|f^1(x)| \cdot \big( 0\vee(t+m)\wedge 1\big). 
\]
Then 
\[
	f_j-g_j\geq \sup_{x\in\mathbb{R}^d}|f_1(x)|1_{[-m,m]^c}
		\quad\mbox{and}\quad
	\int_{\mathbb{R}} f_j\,\ud\overline{\nu}_j-\int_{\mathbb{R}} g_j\,\ud\underline{\nu}_j\leq\frac{\varepsilon}{d},
\]
from which it follows exactly as in the proof of Theorem \ref{thm:main.rep} that $\phi_1(f_n)\downarrow 0=\phi_1(0)$.
\end{proof}

\section{An explicit solution for $f=1_B$}
\label{sec:solution.box}

The optimal transport dualities presented in \cref{thm:main.rep} become really interesting when the primal\footnote{In some parts of the literature on optimal transportation this is called the \textit{primal} problem, see \textit{e.g.} \citet{Villani_2003}, while in other parts this is called the \textit{dual} problem, see \textit{e.g.} \citet{Kellerer_1984}.} problem $\phi(f)$ or $\phi_0(f)$ admits an explicit solution.
Although we cannot expect to deduce an explicit solution for general functions $f$, we will show that $\phi_0(f)$ admits an explicit solution when $f=1_B$, for rectangular sets $B\subset \R^d$.

In order to ease the presentation of the main result in this section we consider in the following the case $d=2$ for a box $B=(-\infty,B_1]\times(-\infty,B_2]$ and finite $I$, \textit{i.e.} $I=\{1,\dots,n\}$. The proofs for the higher-dimensional case ($d>2$) can be obtained by analogous arguments.

\begin{theorem}
\label{thm:sharp.bounds}
The following holds:
\[ 
\max_{\mu\in\mathcal{Q}_0} \mu(B)
	= \min\Big\{ \nu_1((-\infty,B_1]),\nu_2((-\infty,B_2]),
				 \min_{i\in I} \Big\{\bar{\pi}^i + \nu_1((A^i_1,B_1]))+\nu_2((A^i_2,B_2])\Big\}\Big\}. 
\]
\end{theorem}

\begin{figure}[H]
\centering
\begin{minipage}{.33\textwidth}

\begin{tikzpicture}[scale=.85]

\draw[thick,->] (0,0) -- (5.5,0);
\draw[thick,->] (0,0) -- (0,5.5);

\draw[thick] (0,0) -- (4,0) -- (4,4) -- (0,4) -- (0,0);
\draw (3.5,3.5) node {$B$};

\fill[pattern=north west lines, pattern color=blue!30] (0,0) rectangle (5.25,4);

\end{tikzpicture}
\end{minipage}%
\begin{minipage}{.33\textwidth}

\begin{tikzpicture}[scale=.85]

\draw[thick,->] (0,0) -- (5.5,0);
\draw[thick,->] (0,0) -- (0,5.5);

\draw[thick] (0,0) -- (4,0) -- (4,4) -- (0,4) -- (0,0);
\draw (3.5,3.5) node {$B$};

\fill[pattern=north west lines, pattern color=blue!30] (0,0) rectangle (4,5.25);

\end{tikzpicture}
\end{minipage}%
\begin{minipage}{.33\textwidth}

\begin{tikzpicture}[scale=.85]

\draw[thick,->] (0,0) -- (5.5,0);
\draw[thick,->] (0,0) -- (0,5.5);

\draw[thick] (0,0) -- (4,0) -- (4,4) -- (0,4) -- (0,0);
\draw (3.5,3.5) node {$B$};

\draw[thick] (0,0) -- (0,4.75) -- (3,4.75) -- (3,0) -- (0,0);
\draw (2.5,4.5) node {$A$};

\fill[pattern=north west lines, pattern color=blue!30] (0,0) rectangle (3,4.75);
\fill[pattern=north east lines, pattern color=red!30] (3,0) -- (4,0) -- (4,5.25) -- (3,5.25) -- (3,0);

\end{tikzpicture}
\end{minipage}%
\caption{A graphical representation of Theorem \ref{thm:sharp.bounds}.}
\label{Fig:main-theorem}
\end{figure}
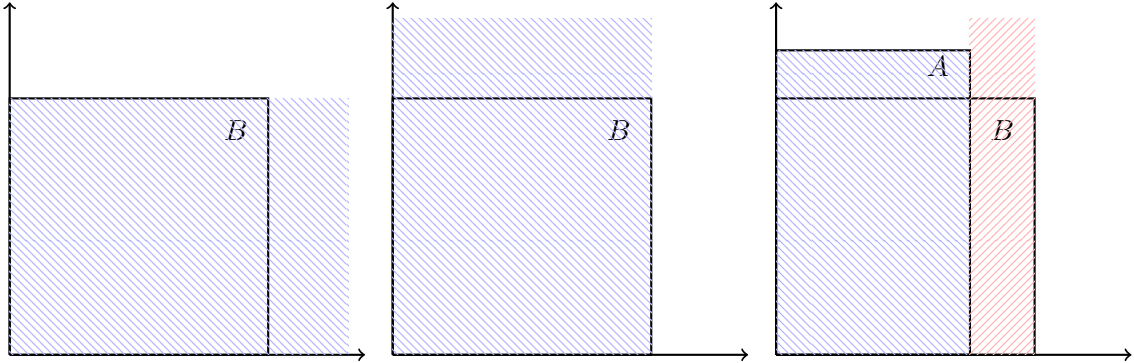

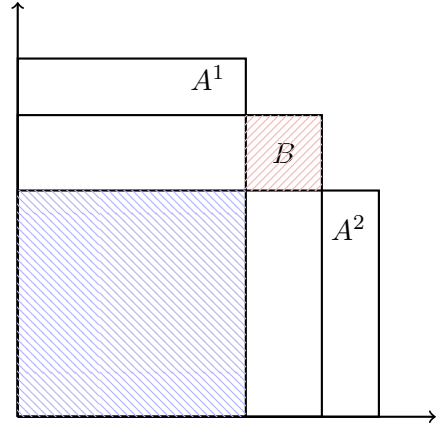
\begin{wrapfigure}{r}{.375\textwidth}
\begin{tikzpicture}[scale=1.]
\draw[thick,->] (0,0) -- (5.5,0);
\draw[thick,->] (0,0) -- (0,5.5);

\draw[thick] (0,0) -- (4,0) -- (4,4) -- (0,4) -- (0,0);
\draw (3.5,3.5) node {$B$};

\draw[thick] (0,0) -- (0,4.75) -- (3,4.75) -- (3,0) -- (0,0);
\draw (2.5,4.5) node {$A^1$};

\draw[thick] (0,0) -- (4.75,0) -- (4.75,3) -- (0,3) -- (0,0);
\draw (4.35,2.5) node {$A^2$};

\fill[pattern=north west lines, pattern color=blue!30] (0,0) rectangle (3,3);
\fill[pattern=north east lines, pattern color=red!30] (3,3) rectangle (4,4);
\end{tikzpicture}
\caption{Non-optimality of super-hedging with two boxes.}
\label{Fig:intuition-2-boxes}
\end{wrapfigure}

Figure \ref{Fig:main-theorem} offers a graphical representation of Theorem \ref{thm:sharp.bounds}.
Let us call `box' an option with payoff $1_B$ with $B\subset\R^2$ and `strip' an option with payoff $1_{J\times\R}$ or $1_{\R\times J}$ with $J\subset\R$.
Then, in the language of mathematical finance, this result states that there are three possible ways to superhedge the box $B$: either using a horizontal strip (left), or a vertical strip (middle), or another box $A$ plus the horizontal and / or vertical strips adjacent to it (right).  

Figure \ref{Fig:intuition-2-boxes} offers an intuitive explanation on why it is not optimal to buy two boxes $A^1$ and $A^2$ in order to superhedge $B$, in the presence of shortselling constraints.
Indeed, in case one buys both $A^1$ and $A^2$, then the 
\begin{tikzpicture} \draw [pattern=north west lines, pattern color=blue!30] (0,0) rectangle (.5,.25); \end{tikzpicture}
shaded region is bought twice incurring unnecessary additional costs, while the 
\begin{tikzpicture} \draw [pattern=north east lines, pattern color=red!30] (0,0) rectangle (.5,.25); \end{tikzpicture}
shaded region is still not hedged.
In order to hedge the latter, a further investment in horizontal and / or vertical strips is required.

\cref{thm:main.rep} applied to $f=1_B$ yields immediately that
\[
\max_{\mu\in\mathcal Q_0} \int 1_B \ud \mu
	= \max_{\mu\in\mathcal Q_0} \mu(B)
	= \phi_0(1_B),
\]
hence we need to show that $\phi_0(1_B)$ admits the following representation:
\[
\phi_0(1_B) = \min\Big\{ \nu_1((-\infty,B_1]),\nu_2((-\infty,B_2]),
				 \min_{i\in I} \Big\{\bar{\pi}^i + \nu_1((A^i_1,B_1]))+\nu_2((A^i_2,B_2])\Big\}\Big\}.
\]

\begin{figure}[H]
\centering
\begin{tikzpicture}[scale=1.]
\draw[thick,->] (0,0) -- (6,0);
\draw[thick,->] (0,0) -- (0,6);

\draw[thick] (0,0) -- (4,0) -- (4,4) -- (0,4) -- (0,0);
\draw (3.65,3.65) node {$B$};

\draw[thick] (0,0) -- (0,5) -- (2,5) -- (2,0) -- (0,0);
\draw (1.65,4.65) node {$A^1$};

\draw[thick] (0,0) -- (5,0) -- (5,2) -- (0,2) -- (0,0);
\draw (4.65,1.65) node {$A^3$};

\draw[thick] (0,0) -- (3,0) -- (3,3) -- (0,3) -- (0,0);
\draw (2.65,2.65) node {$A^2$};

\fill[pattern=north west lines, pattern color=blue!30] (3,0) -- (4,0) -- (4,5.75) -- (3,5.75) -- (3,0);
\draw (3.5,5) node {$S$};

\draw[thick,-] (-.15,5) -- (0,5);
\draw[thick,-] (-.15,4) -- (0,4);
\draw[thick,-] (-.15,3) -- (0,3);
\draw[thick,-] (-.15,2) -- (0,2);

\draw (-.5,5) node {$d_2^4$};
\draw (-.5,4) node {$d_2^3$};
\draw (-.5,3) node {$d_2^2$};
\draw (-.5,2) node {$d_2^1$};

\draw[thick,-] (5,-.15) -- (5,0);
\draw[thick,-] (4,-.15) -- (4,0);
\draw[thick,-] (3,-.15) -- (3,0);
\draw[thick,-] (2,-.15) -- (2,0);

\draw (5,-.5) node {$d_1^4$};
\draw (4,-.5) node {$d_1^3$};
\draw (3,-.5) node {$d_1^2$};
\draw (2,-.5) node {$d_1^1$};

\end{tikzpicture}
\caption{The main setting is illustrated in this figure, where $d_1^{3}=B_1$ hence $i_1=3$.}
\label{Fig:setting}
\end{figure}
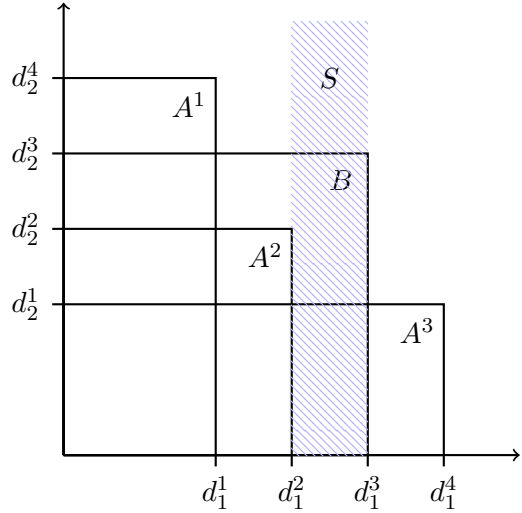

Let us introduce some notation now that will be used in the subsequent proofs; it is illustrated in Figure \ref{Fig:setting}.
Define $D_j:=\{ A^i_j : i\in I \}\cup\{ B_j \}$, for $j=1,2$, and let $D_j=\{d_j^k : k=1,\dots,m_j\}$ be an enumeration such that $d_j^1<d_j^2<\dots<d_j^{m_j}$. Further, define
\[ 
F_j^0:=(-\infty,d_j^1], \quad
F_j^i:=[d_j^i,d_j^{i+1}) \ \text{ for } i=1,\dots,m_j-1,\quad  \text{ and } \quad 
F^{m_j}_j:=(d_j^{m_j},\infty)
\]
for $j=1,2$.
Moreover, let $i_1$ be such that $d_1^{i_1}=B_1$.
In a first step, notice that in the definition of $\phi_0(\1_B)$ we can and will restrict ourselves, without loss of generality, to functions $f_j$ of the form 
\[
f_j(x):=\sum_{i=1}^{m_j} f_j^i \1_{F_j^i}(x),
	\quad \text{where $f_j^i$ are positive constants.}
\] 
We will refer to the functions $f_1$ as ``vertical marginals'' and to the functions $f_2$ as ``horizontal marginals''.

\begin{lemma}
\label{lem:sharp.bounds.reduction}
Let $S:=F_1^{i_1-1}\times\mathbb{R}$.
Then
\begin{align}
\label{eq:psi.equals.min.s}
\phi_0(\1_B) = \min_{s\in\{0,1\}} \Big\{ s\nu_1(F_1^{i_1-1}) + s\phi_0(\1_{B\setminus S}) + (1-s)\eta(\1_B) \Big\},
\end{align}
where
\begin{align}
\label{eq:eta.equals.min}
\eta(\1_B) :=& \, \inf\Big\{ \pi(0,f_2,a) :  f_2(x_2)+\sum_{i\in I} a^i\1_{A^i}(x)= 1
	\ \text{ for all } x\in B \text{ and } f_2,a^i\geq 0 \Big\} \nonumber \\
		   =& \, \min\Big\{ \nu_2((-\infty,B_2]), \min_{i\in I: B_1\leq A_1^i} \Big\{\bar{\pi}^i +\nu_2((A^i_2,B_2])\Big\} \Big\}.
\end{align}
\end{lemma}

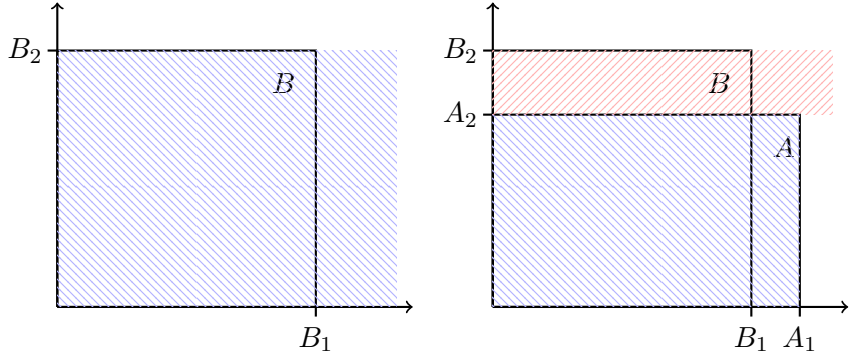
\begin{figure}[H]
\centering
\begin{minipage}{.375\textwidth}

\begin{tikzpicture}[scale=.85]

\draw[thick,->] (0,0) -- (5.5,0);
\draw[thick,->] (0,0) -- (0,4.75);

\draw[thick] (0,0) -- (4,0) -- (4,4) -- (0,4) -- (0,0);
\draw (3.5,3.5) node {$B$};

\draw[thick,-] (4,-.15) -- (4,0);
\draw (4,-.5) node {$B_1$};

\draw[thick,-] (-.15,4) -- (0,4);
\draw (-.5,4) node {$B_2$};

\fill[pattern=north west lines, pattern color=blue!30] (0,0) rectangle (5.25,4);

\end{tikzpicture}
\end{minipage}%
\begin{minipage}{.375\textwidth}

\begin{tikzpicture}[scale=.85]

\draw[thick,->] (0,0) -- (5.5,0);
\draw[thick,->] (0,0) -- (0,4.75);

\draw[thick] (0,0) -- (4,0) -- (4,4) -- (0,4) -- (0,0);
\draw (3.5,3.5) node {$B$};

\draw[thick] (0,0) -- (4.75,0) -- (4.75,3) -- (0,3) -- (0,0);
\draw (4.5,2.5) node {$A$};

\draw[thick,-] (4,-.15) -- (4,0);
\draw (4,-.5) node {$B_1$};

\draw[thick,-] (-.15,4) -- (0,4);
\draw (-.5,4) node {$B_2$};

\draw[thick,-] (4.75,-.15) -- (4.75,0);
\draw (4.75,-.5) node {$A_1$};

\draw[thick,-] (-.15,3) -- (0,3);
\draw (-.5,3) node {$A_2$};

\fill[pattern=north west lines, pattern color=blue!30] (0,0) rectangle (4.75,3);
\fill[pattern=north east lines, pattern color=red!30] (0,3) -- (0,4) -- (5.25,4) -- (5.25,3) -- (0,3);

\end{tikzpicture}
\end{minipage}%
\caption{A graphical representation of the functional $\eta(1_B)$.}
\label{Fig:eta}
\end{figure}

The functional $\eta(1_B)$ is graphically illustrated in Figure \ref{Fig:eta}, and states that there are two ways to superhedge the box $B$ without using the vertical marginals: either using the horizontal strip $1_{\R\times(-\infty,B_2]}$, or using another box $A$ with $A_1\ge B_1$ and, in case $B_2>A_2$, the horizontal strip `above' this box, \textit{i.e.} $1_{\R\times(A_2,B_2]}$.

\begin{proof}
Initially, notice that all optimization problems appearing are finite dimensional linear problems, so that minimizers always exist.

We start by proving \eqref{eq:psi.equals.min.s} and first show that the left hand side is smaller than the right hand side.
Indeed, in case $s=0$, this reduces to the fact that obviously $\phi_0(\1_B)\leq\eta(\1_B)$, since $\phi_0(\1_B)$ is defined as the infimum over a larger set. 
In case $s=1$, let $(f_1,f_2,a)\in\Theta_0(\1_{B\setminus S})$ be optimal --- in the sense that $\pi(f_1,f_2,a)=\phi_0(\1_{B\setminus S})$ --- and notice that one can assume without loss of generality that $f_1^{i_1-1}=0$.
Now define 
\[
\hat{f}_1^i:=
	\begin{cases}
	f_1^i,&\text{if } i \neq i_1-1\\
	1,&\text{else},
	\end{cases}
\]
and it follows that $(\hat{f}_1,f_2,a)\in\Theta_0(\1_B)$. 
By the definition of $\hat{f}_1^i$ it holds that
\[
\pi(\hat{f}_1,f_2,a)
	=\nu_1(F_1^{i_1-1}) + \pi(f_1,f_2,a)
	=\nu_1(F_1^{i_1-1}) + \phi_0(\1_{B\setminus S}),
\]
which shows that $\phi_0(\1_B)\leq \nu_1(F_1^{i_1-1}) + \phi_0(\1_{B\setminus S})$.

In order to prove the reverse inequality, notice that by interchanging two minima it holds
\begin{align}
\label{eq:psi.min.min}
\phi_0(\1_B) = \min_{s\in[0,1]} \Big\{ s\nu_1(F_1^{i_1-1}) + \phi_0^{\setminus i_1-1}(\1_B-s\1_S)\Big\},
\end{align}
where
\[ 
\phi_0^{\setminus i_1-1}(\1_B-s\1_S)
	:=\inf \big\{\pi(f_1,f_2,a) : (f_1,f_2,a)\in\Theta_0(\1_B-s\1_S) \text{ and } 
	f_1^{i_1-1}=0 \big\}. 
\]
Fix some optimal $s$ in \eqref{eq:psi.min.min} and an optimal strategy $(f_1,f_2,a)$ for $\phi_0^{\setminus i_1-1}(\1_B-s\1_S)$.
Since $f_1^{i_1-1}=0$, it follows that 
\[
f_2(x_2)+\sum_{i\in I} a^i\1_{A^i}(x)
	=f_2(x_2)+\sum_{i\in I: B_1\leq A^i_1} a^i\1_{A^i}(x)
	\geq 1-s\quad\text{for all } x\in B\cap S.
\]
Let $t:=\sum_{i\in I : B\subset A^i} a^i$. 
On the one hand, if $t\geq 1-s$, set $\bar{a}^i:=(1-s)a^i/t$ for every $i$ such that $B\subset A^i$, $\bar{a}^i=0$ else, and $\bar{f}_2=0$.
Then $\sum_{i\in I} \bar{a}^i\1_{A^i}(x)=1-s$ for $x\in B$, thus $(0,0,\bar{a})$ is an admissible strategy for $\eta((1-s)\1_B)=(1-s)\eta(\1_B)$.
Further define $\tilde{a}:=a-\bar{a}\geq 0$. 
Then one can check that $(f_1,f_2,\tilde{a})\in\Theta_0(s\1_{B\setminus S})$. 
Therefore
\begin{align}
\phi_0(\1_B)
	&=s\nu_1(F_1^{i_1-1}) + \pi(f_1,f_2,a) \nonumber  \\
	&=s\nu_1(F_1^{i_1-1}) + \pi(f_1,f_2,\tilde{a})+\pi(0,0,\bar{a}) \label{eq:psi(1B).geq.mins} \\
	&\geq \min_{s\in[0,1]} \Big\{ s\nu_1(F_1^{i_1-1}) + s\phi_0(\1_{B\setminus S}) + (1-s)\eta(1_B) \Big\}. \nonumber
\end{align}
Moreover, since the last term is affine in $s$, it follows that the minimum over $s\in[0,1]$ yields the same value as the minimum over $s\in\{0,1\}$.
	
On the other hand, assume that $t<1-s$ and define $\bar{a}^i:=a^i$ for all $i$ such that $B\subset A^i$. 
For notational convenience we assume that $A^i_1\geq B_1$ exactly for $i=1,\dots,m$ and that $B_2> A^\1_2> A^2_2>\dots> A^m_2$; the case where $A^i_2=A^j_2$ for some $i,j\leq m$ works in the same way, but requires additional notation.
Further denote by $k_0$ the index such that $d_2^{k_0}=B_2$ and by $k_i$ the index such that $d_2^{k_i}=A^i_2$, for $i=1,\dots,m$.
Then, for every $i=k_1,\dots,k_0-1$ it needs to hold $f_2^i\geq\bar{f}_2^i:=1-s-t>0$.
Moreover
\begin{align}
\label{eq:induction1}
f_2^i+a^1\geq 1-s-t\quad\text{for }k_2\leq i\leq k_1-1,
\end{align}
\textit{i.e.}~$f_2(x_2)+a^1\geq 1-s-t$ for all $x\in S$ with $x_2\in(A^\1_2,B_2]$.  
Now, there are two possibilities:
\begin{itemize}
\item
If $a^1\geq\bar{a}^1:= 1-s-t$, then set $\bar{f}_2^i:=0$ for $i\leq k_1-1$ and $\bar{a}^i:=0$ for $i=2,\dots,m$. 
Then $(0,\bar{f}_2,\bar{a})$ is an admissible strategy for $\eta(\1_B)$ and $(f_1,\tilde{f}_2,\tilde{a})\in\Theta_0(s\1_{B\setminus S})$, where $\tilde{f}_2:=f_2-\bar{f}_2$ and $\tilde{a}:=a-\bar{a}$.
Hence, it follows from linearity of $\pi$, as in \eqref{eq:psi(1B).geq.mins}, that 
\[
\phi_0(\1_B)\geq \min_{s\in[0,1]}\Big\{ s\nu_1(F_1^{i_1-1}) + s\phi_0(\1_{B\setminus S}) +(1-s)\eta(1_B) \Big\}.
\]
\item
Otherwise, if $\bar{a}^1:=a^1<1-s-t$, define $\bar{f}_2^i:=1-s-t-a^1\leq f_2^i$	for all $k_2\leq i\leq k_1-1$ and set $\tilde{t}:=t+a^1$.
Then 
\[ 
\bar{f}_2(x_2)+\sum_{i\in I : B_1\leq A_1^i} \bar{a}^i\1_{A^i}(x)=1-s\quad\text{for } 
		x\in B \text{ such that } A^\1_2\leq x_2\leq B_2
\]
and necessarily $f_2^i+a^2\geq 1-s-\tilde{t}$ for $k_3\leq i\leq k_2-1$.
This means that the situation is the same as in \eqref{eq:induction1}.
Repeating this procedure at most $m$ times, one finds an admissible strategy $(0,\bar{f}_2,\bar{a})$ for $\eta(\1_B)$. 
Since $(f_1,\tilde{f}_2,\tilde{a})\in\Theta_0(s\1_B)$, where $\tilde{f}_2:=f_2-\bar{f}_2\geq 0$ and $\tilde{a}:=a-\bar{a}\geq0$, it follows from the linearity of $\pi$ that \eqref{eq:psi(1B).geq.mins} holds true.
\end{itemize}

We proceed now with the proof of \eqref{eq:eta.equals.min}.
First notice that for all $i$ with $B\subset A^i$ it holds
\[ 
\eta(\1_B) = \min_{a^i\in[0,1]} \big\{ a^i\bar{\pi}^i + (1-a^i)\eta^{\setminus i}(\1_B)\big\}, \]
where $\eta^{\setminus i}$ is defined as $\eta$, with the additional requirement that $a^i=0$.
Hence $a^i\in\{0,1\}$. 
If $a^i=1$ for some $i$ with $B\subset A^i$, then the proof is complete. 
Otherwise denote by $l$ an element in $\tilde{I}:=\{i\in I : B_1\leq A_1^i\text{ and } A_2^i\leq B_2\}$	such that $A^i_2\leq A^l_2$ for all $i\in\tilde{I}$.
Then $l=d_2^k$ for some $k$ and it necessarily has to hold that $f_2=1$ on $(A^l_2,B_2]$. 
Thus
\[ 
\eta(\1_B) = \nu_2((A^l_2,B_2]) + \eta(\1_{B\setminus S}) 
\]
where $S:=\mathbb{R}\times(A^l_2,B_2]$.
Since $B\setminus S$ is again a box, the claim now follows by induction.
\end{proof}

We are now ready to prove the main result of this section.

\begin{proof}[of Theorem \ref{thm:sharp.bounds}]
Let $S:=F_1^{i_1-1}\times\mathbb{R}$.
If $s=0$ and $s=1$ are both optimizers in \eqref{eq:psi.equals.min.s}, we always chose $s=0$ in order to exclude many pathological cases (see the proof below).

\textit{Case 1:} 
If $s=0$, this means that $\phi_0(\1_B)=\eta(\1_B)$.
However, by \eqref{eq:eta.equals.min} an optimal strategy for $\eta(\1_B)$ consists of either the full horizontal marginal, \textit{i.e.} $f_2=\1_{(-\infty,B_2]}$ and $a=0$, or exactly one box $A^i$ with $B_1\leq A_1^i$ (\textit{i.e.}~$a^i=1$ and $a^j=0$ for $j\neq i$) and the horizontal marginal ``above'' this box, \textit{i.e.} $f_2=\1_{(A^i_2,B_2]}$; see again Figure \ref{Fig:eta}.
Since both strategies are elements of $\Theta_0(\1_B)$, the proof is complete.

\textit{Case 2:} 
If $s=1$, this means that an optimal strategy for $\phi_0(\1_B)$ consists of $f_1^{i_1-1}=1$ plus an optimizer for $\phi_0(\1_{B\setminus S})$. 
If $B\setminus S$ is empty, this means that the optimizer of $\phi_0(\1_B)$ is the full vertical marginal, \textit{i.e.} $f_1=\1_{F_1^{i_1-1}}=\1_{(-\infty,B_1]}$. Otherwise notice that $\hat{B}:=B\setminus S$ is again a (non-empty) box.
Hence one can apply Lemma \ref{lem:sharp.bounds.reduction} again:
Define $\hat{S}:=F_1^{i_1-2}\times\mathbb{R}$ so that
\[
\phi_0(\1_{\hat{B}})
	= \min_{s\in\{0,1\}} \Big\{ s\nu_1(F_1^{i_2-1}) + s\phi_0(\1_{\hat{B}\setminus\hat{S}}) + \eta(\1_{\hat{B}}) \Big\}.
\]
Now, there are again two possibilities: 
\begin{itemize}
\item 	
If $s=0$, \textit{i.e.}~$\phi_0(\1_{\hat{B}})=\eta(\1_{\hat{B}})$, then an optimal strategy for $\eta(\1_{\hat{B}})$ consists of either the full horizontal marginal $f_2=\1_{(-\infty,\hat{B}_2]}=\1_{(-\infty,B_2]}$ only, or exactly one box $A^i$ with $\hat{B}_1\leq A_1^i$ and the horizontal marginal above the box, \textit{i.e.} $f_2=\1_{(A^i_2,\hat{B}_2]}=\1_{(A^i_2,B_2]}$.
We claim that the first case cannot happen, while in the second one it holds $A^i_1=\hat{B}_1$.
Indeed, if $f_2=\1_{(-\infty,B_2]}$ is optimal, then $(0,f_2,0)\in\Theta_0(\1_B)$.
In particular the previous choice $f_1^{i_1-1}=1$ was not optimal. 
Similarly, it follows that in the second case $A^i_1=\hat{B}_1$.

\item 
If $s=1$, then the optimal strategy for $\phi_0(\1_{\hat{B}})$ consists of $f_1^{i_1-1}=f_1^{i_1-2}=1$ plus the optimal one for $\phi_0(\1_{\hat{B}\setminus\hat{S}})$.
\end{itemize}

By induction, it follows that an optimal strategy for $\phi_0(\1_B)$ can take one of the following forms: 
\begin{align*}
\text{either } \   f_1 &= \1_{(\infty,B_1]}, f_2 = 0 \text{ and } a=0, \\
\text{or } \	   f_1 &= 0, f_2=\1_{(-\infty,B_2]}  \text{ and } a=0, \\
\text{or } \	   f_1 &=\1_{(A^i_1,B_1]}, f_2=\1_{(A^i_2,B_2]} \text{ and } a^j = 1 \text{ if } j=i \text{ and } a^j = 0 \text{ else};
\end{align*}
compare again with Figure \ref{Fig:main-theorem}.
This completes the proof.
\end{proof}

\section{Sharpness of the improved upper Fr\'echet--Hoeffding bounds for the classes $\mathcal{F}^{S,\pi}_{\preceq_0}$ and $\mathcal{F}^{S,\pi}_{\preceq_1}$}
\label{sec:proof.main.sharp}

In this section, we show that the improved upper \FH bound is pointwise sharp for the Fr\'echet classes $\mathcal{F}^{S,\pi}_{\preceq_0}$ and $\mathcal{F}^{S,\pi}_{\preceq_1}$ introduced in \eqref{eq:fclass-0} and \eqref{eq:fclass-1} respectively, while the counterexample in the subsequent Appendix \ref{sec:counterexample.not.sharp} shows that the same bound is \textit{not} pointwise sharp for the class $\mathcal{F}^{S,\pi}$.
Hence we will use again the notation of probability theory and will work with distribution functions instead of measures. 
The proof of sharpness for the first class is a straightforward application of the results in Sections \ref{sec:duality} and \ref{sec:solution.box}, while the proof for the second class follows by an explicit construction of a distribution function that belongs to the set $\mathcal{F}^{S,\pi}_{\preceq_1}$ and attains the improved \FH bound.

The question of sharpness or pointwise sharpness of the \FH bounds has a long history in the probability theory literature.
The upper \FH bound is a distribution function itself, hence the bound is actually sharp.
On the other hand, the lower \FH bound is a distribution function, and thus sharp, only in dimension $2$, while in the general case \citet{rueschendorf3} showed that the lower \FH bound is pointwise sharp.
Regarding the improved \FH bounds, \citet{tankov} showed in dimension $2$ that the upper bound is a distribution function, and thus also sharp, in case the set $S$ is decreasing (\textit{i.e.} for $(u_1,u_2),(v_1,v_2)\in S$ holds $(u_1-v_1)(u_2-v_2)\le0$).
This result was later weakened by \citet{bernard}.
On the other hand, \citet{lux2016} showed that in the higher-dimensional case ($d>3$) the upper improved \FH bound is sharp only in trivial cases.
The counterexample in Appendix \ref{sec:counterexample.not.sharp} is therefore surprising, because it shows that once the condition of Tankov is violated the bound is not even pointwise sharp.

\begin{theorem}\label{thm:sharp.zero.order}
The following holds, for every $x\in\mathbb{R}^d$,
\begin{multline*}
\max_{F \in \mathcal{F}^{S,\pi}_{\preceq_0}(F_1^\ast,\dots,F_d^\ast)}  F(x)
	= \min_{i=1,\dots,d} F^\ast_i(x_i) \wedge
		\min\Big\{ \pi_s + \sum_{i=1}^d \big(F_i^\ast(x_i) - F_i^\ast(s_i)\big)^+ : s\in S \Big\}.
\end{multline*}
\end{theorem}

\begin{proof}
Notice that Theorem \ref{thm:sharp.zero.order} is a reformulation of Theorem \ref{thm:sharp.bounds}.
Indeed, the set $\mathcal{Q}_0$ contains all measures induced by the distribution functions in 
$\mathcal{F}_{\preceq_0}^{S,\pi}(F_1^\ast,\dots,F_d^\ast)$, and vice versa.
\end{proof}

\begin{proposition}
\label{prop:sharp.first.order}
Let $S$ be a bounded subset of $\mathbb{R}^d$. Then one has
\begin{equation}\label{sharp5}
	\max_{F\in \mathcal{F}^{S,\pi}_{\preceq_1}(F_1^\ast,\dots,F_d^\ast)} F(x) 
		= \min_{i=1,\dots,d} F^\ast_i(x_i) \wedge \min\{ \pi_s\colon s\in S \text{ such that } x\leq s \},
\end{equation}
for every $x\in\mathbb{R}^d$, where $x\leq s$ whenever $x_i\leq s_i$ for all $i=1,\dots,d$.
\end{proposition}

\begin{proof}
The definition of $\mathcal{F}^{S,\pi}_{\preceq_1}(F_1^\ast,\dots,F_d^\ast)$ immediately implies that the
left hand side (LHS) of \eqref{sharp5}
is smaller than or equal to its right hand side (RHS).
In order to show the reverse inequality, fix $x\in\mathbb{R}^d$ 
and let $r\in\mathbb{R}$ so large that 
\[ r\geq x_j+1 \text{ and } r\geq s_j+1 \text{ for all } j=1,\dots,d, \text{ and } s\in S.\] 
Distinguish between the following two cases.	

\textit{Case 1:}
Assume that the right hand side is attained at $\min_j F_j^*(x_j)$.
Define 
\[ 
G_j(t) := F_j^*(x_j)1_{[x_j, r)}(t) + F_j^*(t)1_{[r,\infty)}(t),
\]
for $t\in\mathbb{R}$, $j=1,\dots,d$, and $F(y)=\min_{j=1,\dots,d} G_j(y_j)$ for $y\in\mathbb{R}^d$.
One can check that $F$ is a cdf, and it holds $F_j(t)=G_j(t)\leq F_j^\ast(t)$ for all $t\in\mathbb{R}$
and $j=1,\dots,d$.
Let $s\in S$. If $x\leq s$, then $x_j\leq s_j\leq r$ for $j=1,\dots,d$ and therefore
$F(s)=\min_j G_j(s)= \min_j F_j^\ast(x_j)=\text{RHS}\leq \pi_s$.
Otherwise, \textit{i.e.}~if there exist some $j^\ast$ such that $s_{j^\ast}<x_{j^\ast}$, one has 
$F(s)\leq G_{j^\ast}(s_{j^\ast})=0\leq \pi_s$.
This shows 
$F\in\mathcal{F}^{S,\pi}_{\preceq_1}(F_1^\ast,\dots,F_d^\ast)$
and sice since $F(x)=\min_j F_j^\ast(x_j)=\text{RHS}$, one obtains that $\text{LHS}\geq \text{RHS}$.

\textit{Case 2:}
Assume that the right hand side is attained at $\pi_{s^\ast}$ for some $s^\ast\in S$.
Define
\[ 
G_j(t) := \pi_{s^\ast} 1_{[x_j, r)}(t) + F_j^*(t)1_{[r,\infty)}(t), 
\]
for $t\in\mathbb{R}$, $j=1,\dots,d$, and $F(y)=\min_j G_j(y_j)$ for $y\in\mathbb{R}^d$.
One can again check that $F$ is a cdf and since $\pi_{s^\ast}\leq F_j^\ast(x_j)$,
one also has $F_j(t)=G_j(t)\leq F_j^\ast(t)$ for all $t\in\mathbb{R}$
and $j=1,\dots,d$.
For $s\in S$ with $x\leq s$, it holds
$F(s)=\min_j G_j(s)= \min_j F_j^\ast(x_j)=\pi_{s^\ast}\leq \pi_s$
since the RHS is attained at $\pi_{s^\ast}$.
Otherwise it holds $F(s)=0$, so that 
$F\in\mathcal{F}^{S,\pi}_{\preceq_1}(F_1^\ast,\dots,F_d^\ast)$.
Since $F(x)=\min_j F_j^\ast(x_j)=\text{RHS}$, one therefore obtains that $\text{LHS}\geq \text{RHS}$.
\end{proof}

\begin{appendix}

\section{The improved upper Fr\'echet--Hoeffding bound is not pointwise sharp for the class $\mathcal{F}^{S,\pi}$}
\label{sec:counterexample.not.sharp}

The following counterexample --- communicated to us by Stephan Eckstein --- illustrates that the improved upper Fr\'echet--Hoeffding bound
\[
	\min\{  F_i^\ast(x_i): i=1,\dots,d\}
		\wedge
	\min\Big\{ \pi_s + \sum_{i=1}^d (F_i^\ast(x_i) - F_i^\ast(s_i) )^+ : s\in S\Big\}
\]
is in general not pointwise sharp for $\mathcal{F}^{S,\pi}(F_1^\ast,\dots,F_d^\ast)$, even in dimension $d=2$. 

\begin{example}
The marginal cdfs are given by
\[ 
	F_1^\ast:=F_2^\ast:=0.1\cdot 1_{[0,1)}+0.3\cdot 1_{[1,2)}+0.35\cdot 1_{[2,3)} + 1_{[3,\infty)},
\]
\textit{i.e.}~$F^\ast_i$ are cdfs of the probability measure $0.1 \delta_{0}+0.2 \delta_{1}+0.05 \delta_{2}+0.65 \delta_{3}$.
Consider the additional information
\[
	S=\{(0,0),(0,2),(2,0),(1,1)\}
		\quad\text{with}\quad
	\pi_{(0,0)}=0\text{ and }\pi_{(0,2)}=\pi_{(2,0)}=\pi_{(1,1)}=0.1.
\]
For the cdf $\hat F$ which corresponds to the probability measure $\sum_{x_1,x_2=0}^3 c_{x_1,x_2}\delta_{(x_1,x_2)}$ with weights given by
\begin{table}[htpb]\centering
\begin{tabular}{|c|c|c|c|c|}
\hline
$c_{x_1,x_2}$ & $x_2=0$ & $x_2=1$& $x_2=2$ & $x_2=3$\\
\hline
$x_1=0$ & 0& 0.05 & 0.05&0\\
\hline 
$x_1=1$ & 0.05 & 0 & 0 &0.15\\
\hline
$x_1=2$ & 0.05 & 0 & 0 &0 \\
\hline
$x_1=3$ & 0 & 0.15 & 0 & 0.5\\
\hline
\end{tabular}
\end{table}

\noindent one can verify that 
\[
	\hat F\in \mathcal{F}^{S,\pi}(F_1^\ast,F_2^\ast) 
		:= \big\{F\in\mathcal{F}(F_1^\ast,F_2^\ast)\colon F(s) = \pi_s \text{ for all } s\in S\big\}.
\] 
This shows that $\mathcal{F}^{S,\pi}(F_1^\ast,F_2^\ast)\neq\emptyset$.
Let $x=(x_1,x_2):=(0,1)$, then the improved Fr\'echet--Hoeffding bound is given by
\[ 	
	\min\{F_1^\ast(0),F_2^\ast(1)\}
		\wedge
	\min\Big\{ \pi_s + \sum_{j=1}^2 (F_j^\ast(x_j) - F_j^\ast(s_j) )^+ : s\in S\Big\} = 0.1,
\]
whereas for $\varphi(u)=1_{\{u\leq x\}}$ it can easily be checked that
\[
	\sup_{F \in \mathcal{F}^{S,\pi}(F_1^\ast,F_2^\ast)} \int \varphi\, \ud F 
		= \sup_{F \in \mathcal{F}^{S,\pi}(F_1^\ast,F_2^\ast)} F(0,1)
		=\hat F(0,1) =0.05.
\]
\end{example}
\end{appendix}

\section{Derivation of the improved upper \FH bounds for the classes $\mathcal{F}^{S,\pi}_{\preceq_0}$ and $\mathcal{F}^{S,\pi}_{\preceq_1}$}
\label{sec:derivation}

In this appendix we show that the improved upper \FH bound is valid for the classes $\mathcal{F}^{S,\pi}_{\preceq_0}$ and $\mathcal{F}^{S,\pi}_{\preceq_1}$.
The derivation uses simple arguments borrowed from copula theory, see \textit{e.g.} \cite{lux2016}.
Let us point out that the sharpness results in \cref{sec:proof.main.sharp} allow us to recover the statements proved below.

\begin{lemma}
Let $G\in \mathcal{F}^{S,\pi}_{\preceq_0}(F_1^*,\dots,F_d^*)$, then we have that
\begin{align}
G(x_1,\dots,x_d) \le \min_{i=1,\dots,d} F^\ast_i(x_i) \wedge \min\Big\{ \pi_s + \sum_{i=1}^d \big(F_i^\ast(x_i) - F_i^\ast(s_i)\big)^+ : s\in S\Big\},
\end{align}	
for all $(x_1,\dots,x_d)\in\R^d$.
Moreover, let $H\in \mathcal{F}^{S,\pi}_{\preceq_1}(F_1^*,\dots,F_d^*)$, then we have that
\begin{align}
H(x_1,\dots,x_d) \le \min_{i=1,\dots,d} F^\ast_i(x_i) \wedge \min\{ \pi_s : s\in S \text{ such that } x\le s \},
\end{align}	
for all $(x_1,\dots,x_d)\in\R^d$.
\end{lemma} 

\begin{proof}
By the definition of the class $\mathcal{F}^{S,\pi}_{\preceq_0}(F_1^*,\dots,F_d^*)$ we have that $G=cF$ where $c\in[0,1]$ and $F$ is a cdf on $\R^d$.
Let us denote by $F_1,\dots,F_d$ the marginals of $F$, then we have immediately that
\begin{align}\label{eq:iFH-0-easy}
G(x_1,\dots,x_d) = c F(x_1,\dots,x_d)
	\le c \min_{i=1,\dots,d} F_i(x_i) = \min_{i=1,\dots,d} c F_i(x_i)
	\le \min_{i=1,\dots,d} F^\ast_i(x_i),
\end{align}
where we have used that $F$ is a cdf for the first inequality and that $cF_i\preceq_0 F_i^\ast$ for the second one.

Using that $F$ is a cdf on $\R^d$ with marginals $F_1,\dots,F_d$, we have the following estimate for any $x_i,s_i\in\R$
\begin{align*}
F(x_1,\dots,x_i,\dots,x_d) - F(x_1,\dots,s_i,\dots,x_d)
	\le \big( F_i(x_i) - F_i(s_i) \big)^+,
\end{align*}
hence, using a telescoping sum, we get the following estimate for any $x,s\in\R^d$
\begin{align}\label{eq:iFH-0-estimate}
F(x) - F(s)	\le \sum_{i=1}^d \big( F_i(x_i) - F_i(s_i) \big)^+.
\end{align}
Therefore, using again that $G=cF$, we have that
\begin{align*}
G(x) &\le  cF(s) + \sum_{i=1}^d \big( cF_i(x_i) - cF_i(s_i) \big)^+\\
	&\le \pi_s + \sum_{i=1}^d \big( F^\ast_i(x_i) - F^\ast_i(s_i) \big)^+,
\end{align*}
where we have used that $cF(s)\le \pi_s$ for all $s\in S$ and that $cF_i(x_i) - cF_i(s_i) \le F^\ast_i(x_i) - F^\ast_i(s_i)$ for all $x_i\le s_i$.
The statement follows by minimizing over all $s\in S$ and combining the outcome with \eqref{eq:iFH-0-easy}.

Now, let $H\in \mathcal{F}^{S,\pi}_{\preceq_1}(F_1^*,\dots,F_d^*)$.
Since $H$ is a cdf on $\R^d$ and using that $F_i\preceq_1 F_i^\ast$, we get immediately that $H(x_1,\dots,x_d) \le \min_{i=1,\dots,d} F^\ast_i(x_i)$.
Moreover, the estimate in \eqref{eq:iFH-0-estimate} is still valid, therefore from the definition of the class $\mathcal{F}^{S,\pi}_{\preceq_1}(F_1^*,\dots,F_d^*)$ we arrive at
\begin{align*}
H(x) \le \pi_s + \sum_{i=1}^d \big( F_i(x_i) - F_i(s_i) \big)^+.
\end{align*}
However, the information available on the marginals, \textit{i.e.} that $F_i\preceq_1 F_i^\ast$, does not allow us to estimate the difference $F_i(x_i) - F_i(s_i)$, and the best we can say is that for $x\le s$ this term collapses to zero.
The statement follows once again my minimizing over all $s\in S$.
\end{proof}

\bibliographystyle{abbrvnat}
\bibliography{bib}

\end{document}